\documentclass[12pt]{amsart}
\usepackage{amssymb}

\newcommand{\eps}{\varepsilon}

\newcommand{\N}{{\mathbb N}}
\newcommand{\C}{{\mathbb C}}
\newcommand{\Z}{{\mathbb Z}}

\newcommand{\tef}{transcendental entire function}

\newcommand{\ef}{entire function}

\newcommand{\mconn}{multiply connected}

\theoremstyle{plain}
\newtheorem{theorem}{Theorem}[section]
\newtheorem{corollary}[theorem]{Corollary}

\newtheorem{lemma}[theorem]{Lemma}
\theoremstyle{definition}

\theoremstyle{remark}

\theoremstyle{problem}

\theoremstyle{example}

\begin{document}


\title[Annular itineraries for entire functions]{Annular itineraries for entire functions}

\author{P. J. Rippon}
\address{Department of Mathematics and Statistics \\
The Open University \\
   Walton Hall\\
   Milton Keynes MK7 6AA\\
   UK}
\email{p.j.rippon@open.ac.uk}

\author{G. M. Stallard}
\address{Department of Mathematics and Statistics \\
The Open University \\
   Walton Hall\\
   Milton Keynes MK7 6AA\\
   UK}
\email{g.m.stallard@open.ac.uk}

\thanks{Both authors were supported by the EPSRC grant EP/H006591/1.}




\begin{abstract}
In order to analyse the way in which the size of the iterates $(f^n(z))$ of a {\tef} $f$ can behave, we introduce the concept of the {\it annular itinerary} of a point $z$. This is the sequence of non-negative integers $s_0s_1\ldots$ defined by
\[
f^n(z)\in A_{s_n}(R),\;\;\text{for }n\ge 0,
\]
where $A_0(R)=\{z:|z|<R\}$ and
\[
A_n(R)=\{z:M^{n-1}(R)\le |z|<M^n(R)\},\;\;n\ge 1.
\]
Here $M(r)$ is the maximum modulus of $f$ and $R>0$ is so large that $M(r)>r$, for $r\ge R$.

We consider the different types of annular itineraries that can occur for any {\tef} $f$ and show that it is always possible to find points with various types of prescribed annular itineraries. The proofs use two new annuli covering results that are of wider interest.
\end{abstract}

\maketitle

\section{Introduction}
\setcounter{equation}{0} Let $f:\C\to \C$ be a {\tef} and denote by
$f^{n},\,n=0,1,2,\ldots\,$, the $n$th iterate of~$f$. The {\it Fatou set} $F(f)$ is
the set of points $z \in \C$ such that $(f^{n})_{n \in \N}$ forms a normal family in
some neighborhood of $z$.  The complement of $F(f)$ is called the {\it Julia set}
$J(f)$ of $f$. The {\it escaping set\,} $I(f)$ of $f$ is defined to be
\[
I(f)=\{z:f^n(z)\to \infty\;\text{ as }n\to\infty\},
\]
and it is known~\cite{E} that we always have $J(f)=\partial I(f)$. An introduction to the properties of these sets can be found in~\cite{wB93}.

One technique that has often been used to understand the nature of the above sets is `symbolic dynamics'. After partitioning $\C$ in some manner and labelling each subset in the partition by a different symbol, we obtain the `itinerary' of a point $z\in\C$ by writing down the sequence of symbols corresponding to the successive subsets of the partition visited by the iterates $(f^n(z))$ of $z$.

This technique was used, for example, in \cite{DT86} to show that, for certain entire functions, the sets $J(f)$ and $I(f)$ contain so-called `Cantor bouquets' of curves, by studying those points with common itineraries with respect to certain natural partitions. This approach was developed and refined in \cite{RRRS} to prove that for a large class of {\ef}s, including all functions of finite order in the much-studied Eremenko-Lyubich class ${\mathcal B}$ (see \cite{EL92}), all points of $I(f)$ can be joined to~$\infty$ by a curve in $I(f)$, a result that is related to a conjecture of Eremenko in \cite{E}.

Symbolic dynamics was also used in \cite{jO12} in order to analyse the dynamical
properties of the components of the complement of the {\it fast escaping set} $A(f)$
for functions~$f$ such that the set $A_R(f)$ has the structure of a `spider's web',
in particular, to show that in this case there are uncountably many components of
$A(f)^c$ of various types. Here
\[
A_R(f) = \{z: |f^{n}(z)| \geq M^n(R), \text{ for } n \in \N\}\quad\text{and}\quad A(f)=\bigcup_{\ell\in\N}f^{-\ell}(A_R(f)),
\]
where $M(r)=M(r,f) = \max_{|z|=r} |f(z)|,\; r > 0$, and $R>0$ is any value such that
$M(r)>r$ for $r\geq R$. See~\cite{RS09a} for the detailed properties of $A(f)$ and
the many families of functions for which $A_R(f)$ is a spider's web. Note that this spider's web
structure cannot occur for functions in the class ${\mathcal B}$ so the results in
\cite{jO12} are complementary to those in \cite{DT86} and \cite{RRRS}.

Here we introduce a new type of partition of the plane which is appropriate for
analysing a key aspect of the dynamical behaviour of {\it any} {\tef}~$f$, namely,
the possible ways in which the size of the iterates $f^n(z)$, $n\ge 0$, can vary. We
define this partition as follows: $A_0(R)=\{z:|z|<R\}$ and
\[
A_n(R)=\{z:M^{n-1}(R)\le |z|<M^n(R)\},\;\;n\ge 1,
\]
where $R>0$ is any value such that $M(r)>r$ for $r\geq R$. Then, for any $z\in\C$, the sequence of non-negative integers $s_0s_1\ldots$ defined by
\begin{equation}\label{itin-def}
f^n(z)\in A_{s_n}(R),\;\;\text{for }n\ge 0,
\end{equation}
will be called the {\it annular itinerary} of $z$, with respect to the partition $\{A_n(R)\}$.

Note that for any annular itinerary $s_0s_1\ldots$ we have $s_{n+1}\le s_n+1$ for $n\in\N$ by the maximum principle. Also, if the value $R$ changes to $R'$, then the annular itinerary, $s_0s_1\ldots$ say, of a point changes to $s'_0s'_1\ldots$ where, for some integer $p$,
\[
s'_n-s_n\in\{p,p+1\},\;\;\text{for }n\ge 0.
\]

Our first result enables us to construct points with many different types of annular itineraries.
\begin{theorem}\label{annuli}
Let $f$ be a {\tef}. There exists $R=R(f)>0$ and a sequence of closed annuli
\[
B_n=\{z:r_n\le |z|\le r'_n\},\;\;n\ge 0,
\]
each of which meets $J(f)$, such that
\begin{equation}\label{prop1}
B_n\subset A_n(R),\;\;\text{for } n\ge 0,
\end{equation}
\begin{equation}\label{prop2}
f(B_n)\supset B_{n+1},\;\;\text{for } n\ge 0,
\end{equation}
and there is a subsequence $B_{n_j}$ such that, for $j\in\N$,
\begin{equation}\label{prop3}
f(B_{n_j})\supset B_{n},\;\; \text{for }0\le n\le n_j, \text{ with at most one exception.}
\end{equation}

Moreover, if $f$ has a {\mconn} Fatou component, then the annuli $B_n$ can be chosen in such a way that \eqref{prop3} holds with no exceptions and the sequence $(n_j)$ consists of those~$n\ge 0$ such that $A_n(R)$ contains a zero of $f$.
\end{theorem}
{\it Remark}\;\;It is clear that the value $R=R(f)$ in Theorem~\ref{annuli} can be chosen to be arbitrarily large.

Theorem~\ref{annuli} gives the following sufficient condition for a sequence $s_0s_1 \ldots$ to be an annular itinerary.
\begin{theorem}\label{itin}
Let $f$ be a {\tef}. There exist $R=R(f)>0$, a sequence $n_j$, $j\in\N$, of positive integers with $n_j\to\infty$ as $j\to\infty$, and a sequence of sets $I_j$, $j\in\N$, where $I_j\subset\{0,\ldots,s_{n_j}\}$ has at most one element, such that:

if $s_0s_1\ldots$ is any sequence of non-negative integers with the property that
\begin{align}\label{itin-prop}
&\text{for }  n\ge 0, \text{ we have } s_{n+1}= s_n+1 \text{ or, in the case when } n=n_j \text{ for some } \\
& j\in\N, \;s_{{n_j}+1}\in \{0,\ldots,s_{n_j}\}\setminus I_j,\notag
\end{align}
then there exists $\zeta\in J(f)$ with annular itinerary $s_0s_1\ldots$ with respect to $\{A_n(R)\}$.

Moreover, if $f$ has a {\mconn} Fatou component, then $I_j=\emptyset$, for $j\in \N$, and the sequence $(n_j)$ consists of those~$n\ge 0$ such that $A_n(R)$ contains a zero of $f$.
\end{theorem}

Theorem~\ref{itin} is related to results in~\cite[Theorem~1.2]{jO12}, where it is
assumed that $A_R(f)$ is a spider's web. It implies that for any {\tef} there exist
many types of prescribed annular itineraries, including those in the following result. Note that a version of Corollary~\ref{itin-types} part~(d) was proved in
\cite[Theorem~1]{RS08c}.

\begin{corollary}\label{itin-types}
Let $f$ be a {\tef}, let $R=R(f)>0$ be the value given by Theorem~\ref{itin} and let $s\in \N$. Then there exist points with the following types of annular itineraries $(s_n)$ with respect to $\{A_n(R)\}$:
\begin{itemize}
\item[(a)] periodic itineraries of all periods with at most one exception such that $s_n\ge s$ for $n\ge 0$;
\item[(b)] uncountably many itineraries such that $(s_n)$ is bounded and $s_n\ge s$ for $n\ge 0$;
\item[(c)] uncountably many itineraries such that $(s_n)$ is unbounded but does not tend to $\infty$, and $s_n\ge s$ for $n\ge 0$;
\item[(d)] uncountably many itineraries such that $s_n\to \infty$ as $n\to\infty$ more slowly than at any given rate.
\end{itemize}
\end{corollary}

By definition a point is in $A(f)$ if and only if its itinerary $s_0s_1\ldots$ with respect to any $\{A_n(R)\}$ satisfies $s_{n+1}=s_n+1$ for sufficiently large~$n$. The existence of such points is well known (see \cite{BH99} or \cite{RS09a}) and also follows from Theorem~\ref{itin}. There are clearly only countably many such itineraries. On the other hand, it follows from our final result that for any given rate of escape that is less than that of `fast escape' there exist uncountably many itineraries for which the corresponding points escape at least at this given rate and in some sense at no faster rate.

\begin{theorem}\label{prescribed}
Let $f$ be a {\tef}. There exists $R_0=R_0(f)>0$ with the property that whenever $(a_n)$ is a positive sequence such that
\begin{equation}\label{aprop-a}
a_n\ge R_0\;\;\text{and}\;\;a_{n+1}\le M(a_n),\;\;\text{for }n\ge 0,
\end{equation}
there exists a point $\zeta\in J(f)$ and a sequence $(n_j)$ with $n_j\to\infty$ as $j\to \infty$ such that
\begin{equation}\label{presc}
|f^n(\zeta)| \geq a_n, \mbox{ for } n \ge 0, \;\mbox{ and } |f^{n_j}(\zeta)| \le M^2(a_{n_j}), \mbox{ for } j \in \N.
\end{equation}

If, in addition,
\begin{equation}\label{notfast}
\text{for each }\ell\in\N \;\text{there exists }n(\ell)\in\N \;\text{such that } a_{n(\ell)+\ell}<M^{n(\ell)}(R_0),
\end{equation}
then there are uncountably many itineraries with respect to $\{A_n(R_0)\}$ that correspond to points $\zeta$ satisfying~\eqref{presc}.
\end{theorem}

{\it Remarks}\;\;1.\; In Theorem~\ref{prescribed} we cannot replace \eqref{presc} by
\begin{equation}\label{twosided}
a_n\le |f^n(\zeta)|\le Ca_n,\quad\text{for sufficiently large } n,
\end{equation}
where $C>1$ is an absolute constant. Indeed, \cite[proof of Theorem~2]{RS08c} shows that if there exist $\zeta\in J(f)$ and $C>1$ such that \eqref{twosided} holds whenever $(a_n)$ is a positive sequence such that $a_n\to\infty$ as $n\to\infty$ and $a_{n+1}\le M(a_n)$ for $n\in\N$, then there must exist constants $c>0$, $d>1$ and $r_0>0$ such that
\begin{equation}\label{minmod}
\text{for all }r\ge r_0\;\;\text{we have } m(s)\le c\;\; \text{for some } s\in
(r,dr),
\end{equation}
where $m(r)$ denotes the {\it minimum modulus} of $f$:
\begin{equation}\label{minmod-def}
m(r)=m(r,f)=\min_{|z|=r} |f(z)|,\quad\text{for }r>0.
\end{equation}
There are many {\tef}s that do not satisfy \eqref{minmod}, in particular those with a {\mconn} Fatou component.

2.\;In \cite{RS12c} we use Theorem~\ref{prescribed} to give a necessary and sufficient condition for a certain subset of $I(f)$ called $Q(f)$, the `quite fast escaping set', to equal $A(f)$. An important special case of the results in~\cite{RS12c} is that $Q(f)=A(f)$ for functions in the class ${\mathcal B}$.\\

We prove the results above by using a method that has its origins in one that we introduced in~\cite{RS08c} in order to construct points that escape arbitrarily slowly. This method involves two complementary annuli covering results, which are significant generalisations of results in~\cite{RS08c} and of independent interest. Both results are given in Section~\ref{preliminary}.

In Section~\ref{MC} we prove a version of Theorem~\ref{annuli} for a {\tef} that has a {\mconn} Fatou component, and in Section~\ref{noMC} we prove a version for a {\tef} that has no {\mconn} Fatou components. The proofs in these two cases are significantly different. In Section~\ref{itin-proof} we prove Theorem~\ref{itin} and Corollary~\ref{itin-types}, and in Section~\ref{prescribed-proof} we prove versions of Theorem~\ref{prescribed} in each of the two cases mentioned above, showing that we obtain a much better upper estimate than the one in \eqref{presc} when there are no {\mconn} Fatou components.

\section{Annuli covering theorems}\label{preliminary}
\setcounter{equation}{0}
To obtain the sequence of annuli in Theorem~\ref{annuli} we prove two theorems about covering properties of annuli, which are considerably stronger than previous results of this nature. Here and later in the paper we use the following basic facts about the maximum modulus.

\begin{lemma}\label{Had}
Let $f$ be a {\tef}. Then
\begin{itemize}
\item[(a)]
$\displaystyle\frac{\log M(r)}{\log r}\to \infty\;\;\text{as } r\to\infty;$
\item[(b)]
there exists $R_1=R_1(f)>0$ such that
\[
M(r^k)\ge M(r)^k,\;\;\text{for } r\ge R_1,\; k>1;
\]
\item[(c)]
$\displaystyle\frac{M(2r)}{M(r)}\to \infty\;\;\text{as } r \to \infty.$
\end{itemize}
\end{lemma}

Part~(a) is a standard property of a {\tef} and part~(b) is a version of Hadamard convexity (see \cite[Lemma~2.1]{RS08}). Part~(c) is well known and easily follows from part~(b), for example.

The proof of our first covering theorem uses the contraction property of the hyperbolic metric. We denote the density of the hyperbolic metric at a point $z$ in a hyperbolic domain $G$ by $\rho_G(z)$
and the hyperbolic distance between~$z_1$ and~$z_2$ in $G$ by
$\rho_G(z_1,z_2)$. Also, for $w\in\C$ and $0<r<s$ we write
\[
B(w,r)=\{z:|z-w|<r\},
\]
\[
A(r,s)=\{z:r<|z|<s\}\;\;\text{and}\;\;\overline A(r,s)=\{z:r\le |z|\le s\},
\]
and we recall that the minimum modulus $m(r)=m(r,f)$ was defined in \eqref{minmod}.
\begin{theorem}\label{Bohr-thm}
Let $f$ be a {\tef}. If $r>0$ and
\begin{equation}\label{msmall}
\text{there exists } s\in (2r,4r)\text{ such that }\;
m(s)\le 1,
\end{equation}
then either
\[
f\left(A(r, 8r)\right)\supset B(0,M(r)),
\]
or there exists $w_1\in B(0, M(r))$ such that
\[
f\left(A(r, 8r)\right)\supset B(0,M(r))\setminus B(w_1,\eps(r)\max\{|w_1|,1\}),
\]
where
\[
\eps(r)=\frac{2C_1}{(C_1\lambda(r))^{1/C_0}},\;\;\text{with}\;\;\lambda(r)=\frac12\left(\frac{M(2r)}{M(r)}-1\right),
\]
and $C_0,C_1>1$ are absolute constants.
\end{theorem}
Note that $\eps(r)\to 0$ as $r \to \infty$, by Lemma~\ref{Had} part~(c), so we have the following corollary of Theorem~\ref{Bohr-thm}.
\begin{corollary}\label{Bohr}
Let $f$ be a {\tef}. There exists $R_2=R_2(f)>0$ such that if $r>R_2$, \eqref{msmall} holds, and $S,S',T,T'$ satisfy
\begin{equation}\label{Rconds}
2<S<S',\;\;T<T'\le M(r)\;\text {and }\;S'\le \tfrac12 T,
\end{equation}
then
\[
f\left(A(r, 8r)\right)\;\;\text{contains}\;\;  A(S,S')\;\;
\text{or}\;\; A(T,T').
\]
\end{corollary}
\begin{proof}[Proof of Theorem~\ref{Bohr-thm}]
Suppose that~\eqref{msmall} holds for some $r>0$, and $s\in (2r, 4r)$ satisfies $m(s)\le 1$. Clearly, $A(r,8r)\supset A(\tfrac12s,2s)$. Then take $\zeta$ and $\zeta'$ such that $|\zeta|=|\zeta'|=s$, and
\begin{equation}\label{zeta}
|f(\zeta)|\le 1\;\;\text{and}\;\;|f(\zeta')|=M(s)\ge M(2r).
\end{equation}
There is an absolute constant $C_0>1$ such that $\rho_{A(s/2,2s)}(\zeta,\zeta')\le \tfrac12\log C_0$.

Now suppose that $f$ omits in $A(s/2, 2s)$ two values $w_1, w_2 \in B(0,M(r))$ such that
\begin{equation}\label{beta}
|w_2-w_1|= \beta \max\{|w_1|,1\},\;\;\text{where } \beta >0.
\end{equation}
By Pick's theorem~\cite[Theorem~4.1]{CG},
\begin{eqnarray}\label{Pick}
\frac12\log C_0&\ge&\rho_{A(s/2,2s)}(\zeta,\zeta')\\
&\ge& \rho_{f(A(s/2,2s))}(f(\zeta),f(\zeta'))\notag\\
&>& \rho_{\C\setminus \{w_1,w_2\}}(f(\zeta),f(\zeta'))\notag\\
&=&\rho_{\C\setminus \{0,1\}}(L(f(\zeta)),L(f(\zeta')))\notag,
\end{eqnarray}
where $L(w)=(w-w_1)/(w_2-w_1)$. Then, by~\eqref{zeta} and \eqref{beta},
\[
|L(f(\zeta))|\le \frac{|f(\zeta)|+|w_1|}{|w_2-w_1|}\le \frac{1+|w_1|}{\beta \max\{|w_1|,1\}}\le \frac{2}{\beta}
\]
and
\[
|L(f(\zeta'))|\ge \frac{|f(\zeta')|-|w_1|}{|w_2|+|w_1|}\ge \frac{M(2r)- M(r)}{2M(r)}= \frac12\left(\frac{M(2r)}{M(r)}-1\right)=\lambda(r).
\]

Now suppose that $\lambda(r)> 2/\beta$. Then, by \eqref{Pick} and \cite[Theorem~9.13]{wH89},
\begin{eqnarray}\label{Hayman}
\frac12\log C_0
&>&\rho_{\C\setminus \{0,1\}}(L(f(\zeta)),L(f(\zeta')))\\
&\ge& \int_{2/\beta}^{\lambda(r)} \frac{dt}{2t(\log C_1 t)}= \frac12\log \frac{\log (C_1\lambda(r))}{\log(2C_1/\beta)},\notag
\end{eqnarray}
where $C_1>1$ is an absolute constant. Thus if
\[
\beta > \eps(r)=\frac{2C_1}{(C_1\lambda(r))^{1/C_0}},
\]
then $\lambda(r)> 2/\beta$, so we obtain a contradiction to \eqref{Hayman}. This proves the result.
\end{proof}

Our second covering theorem is complementary to Corollary~\ref{Bohr}, since we assume here that the modulus of $f$ is greater than~$1$ throughout most of the annulus. Here and subsequently we use the following function in order to simplify notation:
\begin{equation}\label{delta-def}
\delta(r)=\frac{1}{\sqrt{\log r}}.
\end{equation}
\begin{theorem}\label{Harnack}
Let f be a {\tef}, let $k>1$ and let $r$ be such that $r\ge R_1$, where $R_1$ is the constant in Lemma~\ref{Had}, and also such that
\begin{equation}\label{Mr9}
M(r)>r^{9}\;\;\text{and}\;\;\delta(r)<\min\{1/(2\pi),(k-1)/4\}.
\end{equation}

If
\begin{equation}\label{mbig}
m(s)> 1,\quad\text{for } s\in (r^{1+\delta(r)},r^{k-\delta(r)}),
\end{equation}
then the following properties hold.
\begin{itemize}
\item[(a)] We have
 \[
 \log m(s)\ge \left(1-2\pi\delta(r)\right)\log M(s)>0,
 \quad\text{for } s\in [r^{1+2\delta(r)},r^{k-2\delta(r)}].
 \]
 \item[(b)] Let $R=M(r)$ and define $K=K(r)$ by
\[
 R^K=(M(r^{k-2\delta(r)}))^{1-2\pi\delta(r)}.
\]
Then $R > r^{9}$ and
\begin{equation}\label{outside}
f(A(r,r^{k-2\delta(r)}))\supset A\left(R,R^{K}\right)\;\;\text{and}\;\;K\ge k(1-9\delta(r)).
\end{equation}
\item[(c)] If, in addition,
\begin{equation}\label{s-larger}
\delta(r)<\min\{1/(6\pi),(k-1)/(6\pi+1)\},
\end{equation}
then
\begin{equation}\label{inside}
f(A(r^{1+6\pi\delta(r)},r^{k(1-6\pi\delta(r))}))\subset A\left(R^{1+6\pi\delta(R)},R^{K(1-6\pi\delta(R))}\right).
\end{equation}
\end{itemize}
\end{theorem}
\begin{proof}
The proof of part~(a) is based on an application of Harnack's inequality to the harmonic function $u(t)=\log|f(e^t)|$, which is possible by~\eqref{mbig}. Since this part of the proof is identical to \cite[Lemma~5 part~(a)]{RS08c}, we omit the details.

To prove part~(b) first note that, by hypothesis,
\begin{equation}\label{Sands}
R=M(r)>r^{9},\;\;\text{so }\;\frac{\delta(r)}{\delta(R)}= \sqrt{\frac{\log R}{\log r}}> 3.
\end{equation}
Evidently
\begin{equation}\label{inner}
|f(z)|\le M(r)= R,\;\;\text{for } |z|=r,
\end{equation}
and, by part~(a) and the definition of~$K$,
\begin{equation}\label{outer}
m\left(r^{k-2\delta(r)}\right)\ge M\left(r^{k-2\delta(r)}\right)^{1-2\pi\delta(r)}= R^K.
\end{equation}
Also, by Lemma~\ref{Had} part~(b),
\begin{align}\label{Kandk}
R^K&=M\left(r^{k-2\delta(r)}\right)^{1-2\pi\delta(r)}\\
&\ge M\left(r\right)^{(k-2\delta(r))(1-2\pi\delta(r))}\notag\\
&\ge M\left(r\right)^{k(1-9\delta(r))}\notag\\
&=R^{k(1-9\delta(r))}.\notag
\end{align}
Taken together,~\eqref{inner},~\eqref{outer} and~\eqref{Kandk} imply~\eqref{outside}.

To prove part~(c), we note that $6\pi\delta(r)<1$ and
$1+6\pi\delta(r)\in(1+\delta(r),k-\delta(r))$, by \eqref{s-larger}. Thus, by
part~(a), Lemma~\ref{Had} part~(b), and~\eqref{Sands},
\begin{align}\label{inner1}
m\left(r^{1+6\pi\delta(r)}\right)
&\ge M\left(r^{1+6\pi\delta(r)}\right)^{1-2\pi\delta(r)}\\
& \ge M\left(r\right)^{(1+6\pi\delta(r))(1-2\pi\delta(r))}\notag\\
& \ge M\left(r\right)^{1+2\pi\delta(r)}\notag\\
& \ge M\left(r\right)^{1+6\pi\delta(R)}\notag\\
&= R^{1+6\pi\delta(R)}\notag.
\end{align}
Also, by Lemma~\ref{Had} part~(b),~\eqref{s-larger},~\eqref{Sands}, and the facts
that $k(1-2\delta(r))<k-2\delta(r)$ and $6\pi\delta(r)<1$,
\begin{align}\label{outer1}
M\left(r^{k(1-6\pi\delta(r))}\right)
&\le M\left(r^{k(1-2\delta(r))}\right)^{\frac{1-6\pi\delta(r)}{1-2\delta(r)}}\\
&< R^{\frac{K}{1-2\pi\delta(r)}\frac{1-6\pi\delta(r)}{1-2\delta(r)}}\notag\\
&\le R^{K(1-(6\pi-2-2\pi)\delta(r))}\notag\\
&\le R^{K(1-6\pi\delta(R))}\notag.
\end{align}
Taken together, the estimates~\eqref{inner1} and~\eqref{outer1} imply~\eqref{inside}, since $f$ has no zeros in the annulus $A(r^{1+6\pi\delta(r)},r^{k(1-6\pi\delta(r))})$.
\end{proof}

\section{Proof of Theorem~\ref{annuli}: {\mconn} Fatou components}\label{MC}
\setcounter{equation}{0}
In this section we prove a version of Theorem~\ref{annuli} in which it is assumed that~$f$ has a {\mconn} Fatou component. Recall from \cite[Theorem~3.1]{iB84} that if~$U$ is a {\mconn} Fatou component of a {\tef}~$f$, then the Fatou components $U_n=f^n(U)$, $n\ge 0$, form a sequence of eventually nested `ring-like' domains, in the sense that:
\begin{itemize}
\item[(a)]
each $U_n$ is bounded and {\mconn};
\item[(b)]
$U_{n+1}$ surrounds $U_n$ for sufficiently large $n$;
\item[(c)]
{\rm dist}$(U_n,0)\to\infty$ as $n\to\infty$.
\end{itemize}
Such {\mconn} Fatou components $U_n$ are known to contain large annuli centred at~$0$, as shown by the following result from \cite[Theorem~1.2]{BRS11} which strengthens an earlier result of Zheng \cite{jhZ02}.
\begin{lemma}\label{large-annuli}
Let $f$ be a {\tef} with a {\mconn} Fatou component $U$. For each $z_0 \in U$ and each open set $D \subset U$ containing $z_0$,
 there exists $\alpha > 0$ such that, for sufficiently large $n \in \N$,
    \[
     U_n \supset f^n(D) \supset A(|f^n(z_0)|^{1-\alpha}, |f^n(z_0)|^{1+\alpha}).
    \]
\end{lemma}

We use Lemma~\ref{large-annuli} to prove a general result about the existence of sequences of absorbing annuli in such {\mconn} Fatou components, which in a sense strengthens the results in~\cite{BRS11}. Recall that $\delta(r) =1/\sqrt{\log r}$.

\begin{lemma}\label{absorb}
Let $f$ be a {\tef} with a {\mconn} Fatou component. Then there exist sequences $(r_n)$ and $(k_n)$, with $r_n > 0$ and $k_n>1$, for $n\ge 0$, such that the annuli
\begin{equation}\label{An}
A_n=A(r_n,r_n^{k_n}),\;\;n\ge 0,
\end{equation}
and
\begin{equation}\label{A'n}
A'_n=A(r_n^{1+6\pi\delta_n},r_n^{k_n(1-6\pi\delta_n)}),\;\;\text{where }\delta_n=\delta(r_n),\;n\ge 0,
\end{equation}
have the following properties.
\begin{itemize}
\item[(a)] For $n\ge 0$,
\begin{equation}\label{propn-3}
f(A_n)\supset A_{n+1}
\end{equation}
and
\begin{equation}\label{propn-4}
f(A'_n)\subset A'_{n+1}.
\end{equation}
\item[(b)]
 The annuli $A'_n$, $n\ge 0$, lie in distinct {\mconn} Fatou components~$U_n$ of $f$ and $f(U_n)=U_{n+1}$, for $n\ge 0$.
\item[(c)] For $n\ge 0$,
\begin{equation}\label{propn-1}
r_{n+1}= M(r_n) > r_n^{16},
\end{equation}
\begin{equation}\label{propn-1a}
r_{n+1}^{k_{n+1}}=M(r_n^{k_n-2\delta_n})^{1-2\pi\delta_n},
\end{equation}
and
\begin{equation}\label{propn-2}
1+20\delta_0\le k_0<\frac{\log r_1}{\log r_0},\;\;k_{n+1}\ge k_n(1-9\delta_n)\ge 1+20\delta_{n+1} \;\;\text{and}\;\; r_n^{k_n}<r_{n+1}.
\end{equation}
\end{itemize}
\end{lemma}
\begin{proof}
By Lemma~\ref{large-annuli}, there exists a {\mconn} Fatou component $U$ of $f$, a point $z_0\in U$ and a constant $\alpha>0$ such that the Fatou components $U_n=f^n(U)$, $n\ge 0$, satisfy
\[
U_n\supset A(|f^n(z_0)|^{1-\alpha},|f^n(z_0)|^{1+\alpha}),\;\;\text{for }n\ge 0.
\]

Therefore, by replacing $U$ by $U_n$ for a suitably large $n$ and relabeling, we can take $r_0>0$ such that
\begin{equation}\label{M(r)16}
M(r)>r^{16},\;\;\text{for } r\ge r_0.
\end{equation}
\begin{equation}\label{A0}
U\supset A(r_0,r_0^{k_0}),
\end{equation}
where
\begin{equation}\label{delta0}
1+20\delta_0\le k_0<\frac{\log M(r_0)}{\log r_0}\;\;\text{and}\;\;\delta_0=\delta(r_0)\le \frac{1}{80},
\end{equation}
and also, by property~(c) before Lemma~\ref{large-annuli},
\begin{equation}\label{fnlarge}
|f^n(z)|>1,\;\;\text{for } z\in A(r_0,r_0^{k_0}),\;n\ge 0.
\end{equation}

Suppose now that for some $m\ge 0$ the sequences $r_0,\ldots,r_{m}$ and $k_0,\ldots,k_{m}$ have been chosen so that they satisfy \eqref{propn-3}, \eqref{propn-4}, \eqref{propn-1}, \eqref{propn-1a} and \eqref{propn-2}, for $n=0,\ldots, m-1$. We show that $r_{m+1}$ and $k_{m+1}$ can be chosen so that these properties also hold for $n=m$.

By~\eqref{propn-2}, \eqref{M(r)16} and the fact that $6\pi+1<20$, we can apply Theorem~\ref{Harnack} with $r=r_m$ and $k=k_m$ to deduce that
\begin{equation}\label{contains-RK-0}
f\left(A_m\right)\supset A\left(R,R^{K}\right),
\end{equation}
where
\begin{equation}\label{outside-0}
R=M(r_m)>r_{m}^{16},\;\;R^K=M(r_m^{k_m-2\delta_m})^{1-2\pi\delta_m}\;\;\text{and}\;\; K\ge k_{m}(1-9\delta_m),
\end{equation}
and
\begin{equation}\label{inside-0}
f(A(r_m^{1+6\pi\delta_m},r_m^{k_m(1-6\pi\delta_m)}))\subset A\left(R^{1+6\pi\delta(R)},R^{K(1-6\pi\delta(R))}\right).
\end{equation}

We now take
\[
r_{m+1}=R,\;\; \delta_{m+1}=\delta(R)\;\;\text{and}\;\; k_{m+1}=K.
\]
Then, by \eqref{contains-RK-0},
\[
f\left(A_m\right)\supset A\left(r_{m+1},r_{m+1}^{k_{m+1}}\right),
\]
by~\eqref{outside-0},
\[
k_{m+1}\ge k_m(1-9\delta_m),
\]
and, by \eqref{M(r)16},
\[
\frac{\delta_j}{\delta_{j+1}}=\frac{\sqrt{\log r_{j+1}}}{\sqrt{\log r_j}}=\sqrt{\frac{\log M(r_j)}{\log r_j}}\ge 4,\;\;\text{for }j=0,\ldots,m.
\]
Therefore, by \eqref{delta0},
\[
k_{m+1} \ge (1+20\delta_0)\prod_{j=0}^m(1-9\delta_j)\ge (1+20\delta_0)(1-12\delta_0)\ge 1+5\delta_0\ge 1+20\delta_{m+1}.
\]
Finally, the inequality $r_m^{k_m}<M(r_m)=r_{m+1}$ holds in the case $m=0$, by \eqref{delta0}, and we deduce it in the case $m\ge 1$ from \eqref{propn-1a} and \eqref{propn-2} with $n=m-1$, since these imply that
\[
r_m^{k_m}\le M(r_{m-1}^{k_{m-1}})<M(r_m)=r_{m+1}.
\]
We have now checked that \eqref{propn-3}, \eqref{propn-4}, \eqref{propn-1}, \eqref{propn-1a} and \eqref{propn-2} all hold with $n=m$, and by induction this completes the proofs of parts~(a) and~(c).

To prove part~(b), note that \eqref{An}, \eqref{A'n}, \eqref{propn-3} and \eqref{A0} imply that for $n\ge 0$ we have $A'_n\subset A_n\subset f^n(U)$, and the sets $U_n=f^n(U)$ form disjoint Fatou components by the theorem of Baker \cite[Theorem~3.1]{iB84} mentioned before Lemma~\ref{large-annuli}.
\end{proof}
We now use Lemma~\ref{absorb} to prove the stronger version of Theorem~\ref{annuli} that holds when $f$ has a {\mconn} Fatou component. Here, for a closed annulus $B=\overline A(r,s)$, $0<r<s$, we denote by $\partial_{\rm inn}B$ and $\partial_{\rm out}B$ the inner and outer boundary components of $B$, respectively.
\begin{theorem}\label{annuli-MC}
Let $f$ be a {\tef} with a {\mconn} Fatou component. Then there exists $R=R(f)>0$, a
sequence $(U_n)$ of distinct {\mconn} Fatou components of $f$ such that
$f(U_n)=U_{n+1}$, for $n\ge 0$, and a sequence of closed annuli
\[
B_n=\overline A(t_n,t'_n),\;\;n\ge 0,
\]
such that
\[
\partial_{\rm inn}B_n\subset U_n\;\;\text{and}\;\;\partial_{\rm out} B_n\subset
U_{n+1},\;\;\text{for } n\ge 0,
\]
so each $B_n$ meets $J(f)$,
\begin{equation}\label{prop1-b}
[t_{n+1},t'_{n+1}]\subset [M(t_n),M(t'_n)],\;\;\text{for } n\ge 0,
\end{equation}
\begin{equation}\label{prop1-bb}
B_n\subset A_n(R),\;\;\text{for } n\ge 0,
\end{equation}
\begin{equation}\label{prop3-b}
f(B_n)\supset B_{n+1},\;\;\text{for } n\ge 0,
\end{equation}
and there is a subsequence $(B_{n_j})$ such that
\begin{equation}\label{subseq1}
f(B_{n_j})\supset B_{n},\;\;\text{for }j\in\N,\;0\le n\le n_j.
\end{equation}

Also, the sequence $(n_j)$ consists of those~$n\ge 0$ such that $A_n(R)$ contains a zero of $f$.
\end{theorem}
\begin{proof}
Let $(r_n)$, $(k_n)$, $(A_n)$, $(A'_n)$ and $(U_n)$ be the sequences given by
Lemma~\ref{absorb}; in particular,
\begin{equation}\label{A-sub}
A'_n\subset U_n\;\;\text{and}\;\; f(A'_n) \subset A'_{n+1},\;\;\text{for }n\ge 0,
\end{equation}
by \eqref{propn-4}.

Now define the sequence of closed annuli $(B_n)$ as follows:
\[
B_n=\overline A(t_n,t'_n),\;\;\text{for }n\ge 0,
\]
where
\begin{equation}\label{s-def}
t_n=r_n^{k_n(1-6\pi\delta_n)}\;\;\text{and}\;\;
t'_n=r_{n+1}^{1+6\pi\delta_{n+1}},\;\;\text{for }n\ge 0;
\end{equation}
that is, $B_n$ is the annulus lying between $A'_n$ and $A'_{n+1}$. Since $A'_n$ and $A'_{n+1}$ lie in distinct Fatou components, it follows that $B_n\cap J(f)\ne \emptyset$, for $n\ge 0$.

Also, by \eqref{A-sub}, we have
\begin{equation}\label{boundary-Bn}
f(\partial_{\rm inn}B_n)\subset \overline{A'_{n+1}}\;\;\text{and}\;\;f(\partial_{\rm out}B_n)\subset \overline{A'_{n+2}},\;\;\text{for }n\ge 0,
\end{equation}
so
\[
f(B_n)\supset B_{n+1},\;\;\text{for }n\ge 0,
\]
which proves \eqref{prop3-b}.

We also have in this case, by \eqref{A'n} and \eqref{boundary-Bn}, that
\begin{equation}\label{t-t'-a}
t'_n< M(t_n)\le r_{n+1}^{k_{n+1}(1-6\pi\delta_{n+1})}=t_{n+1},\;\;\text{for } n\ge 0,
\end{equation}
and, by \eqref{propn-2}, \eqref{s-def} and \eqref{boundary-Bn}, that
\begin{equation}\label{t-t'-b}
t_{n+1}<r_{n+1}^{k_{n+1}}<M(r_{n+1})=r_{n+2}<t'_{n+1}<M(t'_n),\;\;\text{for } n\ge 0.
\end{equation}
Together, \eqref{t-t'-a} and \eqref{t-t'-b} prove \eqref{prop1-b}.

Now \eqref{t-t'-a} implies that there exists $R=R(f)>0$ such that
\[
t'_0<R<t_1<t'_1<M(R),
\]
so $B_0\subset A_0(R)$ and $B_1\subset A_1(R)$, and we deduce by \eqref{prop1-b} that
\[
M^{n-1}(R)<t_n<t'_n<M^n(R),\;\;\text{so}\;\;B_n\subset A_n(R),\;\;\text{for }n\in\N,
\]
which proves \eqref{prop1-bb}.

To prove the final statement, let
\[
B'_n=A'_n\cup B_n\cup A'_{n+1}, \;\;n\ge 0,
\]
and let $F_n$ denote the component of $\C\setminus B'_n$ that contains~0. Then, by \eqref{A-sub},
\[
\partial f(B'_n)\subset f(\partial B'_n)\subset B'_{n+1}, \;\;\text{for } n\ge 0,
\]
and hence, for each $n\ge 0$, we have exactly one of the following possibilities:
\begin{equation}\label{onto}
f(B'_n)\subset B'_{n+1},
\end{equation}
or
\begin{equation}\label{notonto}
f(B'_n)\supset F_{n+1},\;\;\text{so }\; f(B_n)\supset F_{n+1}.
\end{equation}

If~(\ref{onto}) holds for all $n\ge N$, say, then each $B'_n$, $n\ge N$, is contained in the Fatou set of $f$, by Montel's theorem, and this contradicts the fact that each $B_n$ and hence each $B'_n$ meets $J(f)$. Thus there is a strictly increasing sequence $n_j\ge 0$ such that~(\ref{notonto}) holds for all $n=n_j$, $j\in\N$, which gives \eqref{subseq1}. By replacing~$R$ by $M^n(R)$ for some $n$, we can assume that $n_0=0$. By \eqref{onto} and \eqref{notonto}, this sequence $(n_j)$ consists of those~$n\ge 0$ such that $B_n$ contains a zero of~$f$. Hence, by \eqref{prop1-bb} and \eqref{A-sub}, this sequence consists of those $n\ge 0$ such that $A_n(R)$ contains a zero of~$f$.
\end{proof}

\section{Proof of Theorem~\ref{annuli}: no {\mconn} Fatou components}\label{noMC}
\setcounter{equation}{0}

The proof of Theorem~\ref{annuli} in the case of {\it no} {\mconn} Fatou components is somewhat more complicated and is based on the following lemma. Recall again that $\delta(r)=1/\sqrt{\log r}$.

\begin{lemma}\label{annuli-N}
Let $f$ be a {\tef} with no {\mconn} Fatou components. There exists $R_3=R_3(f)>0$ such that for all $r_0\ge R_3$ and $k_0=1+20\delta(r_0)$ there exist finite sequences $r_n>0$ and $k_n>1$, $n=0,\ldots,N$, where $N\ge 0$, such that the annuli
\[
A_n=A(r_n,r_n^{k_n}),\;\;n=0,\ldots, N,
\]
and
\[
A'_n=A(r_n^{1+6\pi\delta_n},r_n^{k_n(1-6\pi\delta_n)}),\;\;\text{where }\delta_n=\delta(r_n),\;n=0,\ldots, N,
\]
have the following properties.
\begin{itemize}
\item[(a)] For $n=0,\ldots,N-1$,
\begin{equation}\label{prop-3}
f(A_n)\supset A_{n+1}
\end{equation}
and
\begin{equation}\label{prop-4}
f(A'_n)\subset A'_{n+1}.
\end{equation}
\item[(b)]
For any $S,S',T,T'$ that satisfy
\begin{equation}\label{R-conds-1}
2<S<S',\;\;T<T'\le M(\tfrac13 r_{N}^{1+\delta_N})\;\;\text {and }\;S'\le \tfrac12 T,
\end{equation}
we have
\[
f(A_{N})\;\;\text{contains}\;\;  A(S,S')\;\;
\text{or}\;\; A(T,T').
\]
\item[(c)] For $n=0,\ldots,N-1$,
\begin{equation}\label{prop-1}
r_{n+1}= M(r_n) > r_n^{16},
\end{equation}
\begin{equation}\label{prop-1a}
r_{n+1}^{k_{n+1}}=M(r_n^{k_n-2\delta_n})^{1-2\pi\delta_n},
\end{equation}
\begin{equation}\label{prop-2}
k_{n+1}\ge k_n(1-9\delta_n)\ge 1+5\delta_0 \ge 1+20\delta_{n+1},
\end{equation}
and
\begin{equation}\label{prop-2a}
M(r_n^{k_n})>r_{n+1}^{k_{n+1}}>r_{n+1}=M(r_n)>(r_n^{k_n})^2.
\end{equation}
Also,
\begin{equation}\label{prop-2b}
M(r_N)>(r_N^{k_N})^2.
\end{equation}
\end{itemize}
\end{lemma}
\begin{proof}
We construct the annuli $A_n$ by using Theorem~\ref{Harnack} in a way that is similar to the proof of Lemma~\ref{absorb} but rather more complicated.

We take $R_3\ge \max\{R_1^2,R_2\}$, where $R_1=R_1(f)$ is the constant in Lemma~\ref{Had} and $R_2=R_2(f)$ is the constant in Corollary~\ref{Bohr}, and also such that
\begin{equation}\label{M(r)}
M(r)> r^{16}\;\;\text{and}\;\;\sqrt{\log r}\ge 80,\;\;\text{for }r \ge R_3.
\end{equation}
Given $r_0\ge R_3$, we define the annulus $A_0=A(r_0,r_0^{k_0})$, where $k_0=1+20\delta_0$ with $\delta_0=\delta(r_0)$. Suppose that, for some $m\ge 0$, the sequences $r_0, \ldots, r_{m}$ and $k_0, \ldots, k_{m}$ have been chosen so that they satisfy \eqref{prop-3}, \eqref{prop-4}, \eqref{prop-1}, \eqref{prop-1a}, \eqref{prop-2} and \eqref{prop-2a}, for $n=0,\ldots, m-1$.

We now show that if the condition
\begin{equation}\label{m-big}
m(s)> 1,\quad\text{for all } s\in
(r_{m}^{1+\delta_{m}},r_{m}^{k_{m}-\delta_{m}}),
\end{equation}
holds, then $r_{m+1}$ and $k_{m+1}$ can be chosen so that these properties also hold with $n=m$.

By \eqref{prop-2}, \eqref{M(r)} and \eqref{m-big}, we can apply Theorem~\ref{Harnack} with $r=r_m$ and $k=k_m$ to deduce that
\begin{equation}\label{contains-RK}
f\left(A_m\right)=f(A(r_m,r_m^{k_m}))\supset A\left(R,R^{K}\right),
\end{equation}
where
\begin{equation}\label{outside1}
R=M(r_m)>r_{m}^{16},\;\;R^K=M(r_m^{k_m-2\delta_m})^{1-2\pi\delta_m}\;\;\text{and}\;\; K\ge k_{m}(1-9\delta_m),
\end{equation}
and
\begin{equation}\label{inside1}
f(A(r_m^{1+6\pi\delta_m},r_m^{k_m(1-6\pi\delta_m)}))\subset A\left(R^{1+6\pi\delta(R)},R^{K(1-6\pi\delta(R))}\right).
\end{equation}

So if we take
\[
r_{m+1}=R,\;\; \delta_{m+1}=\delta(R)\;\;\text{and}\;\; k_{m+1}=K,
\]
then \eqref{contains-RK}, \eqref{outside1} and \eqref{inside1} imply that the properties \eqref{prop-3}, \eqref{prop-4}, \eqref{prop-1} and \eqref{prop-1a} all hold with $n=m$.

To deduce \eqref{prop-2} in the case $n=m$ first note that, by \eqref{M(r)},
\[
\delta_0\le \frac{1}{80}\;\;\text{and}\;\;\frac{\delta_j}{\delta_{j+1}}=\frac{\sqrt{\log r_{j+1}}}{\sqrt{\log r_j}}=\sqrt{\frac{\log M(r_j)}{\log r_j}}\ge 4,\;\;\text{for }j=0,\ldots, m.
\]
Thus, by~\eqref{prop-2} with $n=0,\ldots, m-1$,~\eqref{outside1}, and the fact that $k_{m+1}=K$,
\[
k_{m+1}\ge k_0\prod_{j=0}^m(1-9\delta_j)\ge (1+20\delta_0)(1-12\delta_0)\ge 1+5\delta_0\ge 1+20\delta_{m+1},
\]
as required.

The first two inequalities in \eqref{prop-2a} with $n=m$ follow immediately from
\eqref{prop-1a} and \eqref{prop-2}. Also, the inequality
$(r_m^{k_m})^2<M(r_m)=r_{m+1}$ holds in the case $m=0$, by \eqref{M(r)}, and we can deduce this inequality in the case $m\ge 1$ from \eqref{prop-1a}, \eqref{prop-2a} with $n=m-1$, and
Lemma~\ref{Had} part~(b) (applied with $r=r_{m-1}^{k_{m-1}}$ and $k=2$) as follows:
\begin{eqnarray}\label{N-case}
(r_m^{k_m})^2&=&(M(r_{m-1}^{k_{m-1}-2\delta_{m-1}})^{1-2\pi\delta_{m-1}})^2\le (M(r_{m-1}^{k_{m-1}}))^2\\
&\le &M((r_{m-1}^{k_{m-1}})^2)< M(r_m)= r_{m+1}.\notag
\end{eqnarray}
Hence \eqref{prop-3}, \eqref{prop-4},~\eqref{prop-1}, \eqref{prop-1a},~\eqref{prop-2} and \eqref{prop-2a} all hold for $n=m$.

We now observe that the condition in \eqref{m-big} cannot hold for all $m\ge 0$. For if this were the case, then we could construct a sequence of annuli $(A_n)_0^{\infty}$ satisfying \eqref{prop-3}, \eqref{prop-4}, \eqref{prop-1}, \eqref{prop-1a} and \eqref{prop-2}, and property \eqref{prop-4} would then imply that~$f$ has {\mconn} Fatou components, contrary to our hypothesis.

Therefore for some $N\ge 0$ the condition in \eqref{m-big} holds for $m=0,\ldots ,N-1$, and the annuli $A(r_n,r_n^{k_n})$, $n=0,\ldots, N-1$, exist satisfying properties~(a) and~(c), including \eqref{prop-2b} by \eqref{N-case}, but \eqref{m-big} fails for $m=N$; that is,
\begin{equation}\label{rho}
\text{there exists } s\in (r_N^{1+\delta_N}, r_{N}^{k_N-\delta_N})
\text{ such that }\;m(s)\le 1.
\end{equation}
Then
\begin{equation}\label{s-int}
s\in (\tfrac23s,\tfrac43s)\subset(\tfrac13s,\tfrac83s)
\subset (\tfrac13 r_N^{1+\delta_N},\tfrac83 r_N^{k_N-\delta_N})\subset (r_N,r_N^{k_N}),
\end{equation}
since $r_N^{\delta_N}=\exp(\sqrt{\log r_N})\ge 3$. Therefore, since $m(s)\le 1$, we can apply Corollary~\ref{Bohr} with $r=\tfrac13 s$, to deduce that if $S,S',T,T'$ satisfy
\[
2<S<S',\;\;T<T'\le M(\tfrac13 r_{N}^{1+\delta_N})\;\;\text {and }\;S'\le \tfrac12 T,
\]
then
\[
f(A(\tfrac13s,\tfrac83s))\;\;\text{contains}\;\;  A(S,S')\;\;
\text{or}\;\; A(T,T'),
\]
so, by \eqref{s-int},
\[
f(A_{N})\;\;\text{contains}\;\;  A(S,S')\;\;
\text{or}\;\; A(T,T'),
\]
as required.
\end{proof}
We now use Lemma~\ref{annuli-N} to prove Theorem~\ref{annuli} in the case when $f$ has no {\mconn} Fatou components.
\begin{theorem}\label{annuli-noMC}
Let $f$ be a {\tef} with no {\mconn} Fatou components. There exists $R=R(f)>0$ and a sequence of closed annuli
\begin{equation}\label{Bn-form}
B_n=\overline A(r_n,r^{k_n}_n),\;\;n\ge 0 ,
\end{equation}
each of which meets $J(f)$, such that
\begin{equation}\label{prop1-bbb}
B_n\subset A_n(R),\;\;\text{for } n\ge 0,
\end{equation}
and
\begin{equation}\label{prop2-a}
f(B_n)\supset B_{n+1},\;\;\text{for } n\ge 0,
\end{equation}
and there is a subsequence $B_{n_j}$ such that, for $j\in\N$,
\begin{equation}\label{subseq2}
f(B_{n_j})\supset B_{n},\;\; \text{for }0\le n\le n_j, \text{ with at most one exception.}
\end{equation}
\end{theorem}
\begin{proof}
Let $R_4=R_4(f)>0$ be so large that
\begin{equation}\label{M(r)-a}
M(r)> r^{10000}\;\;\text{and}\;\;\sqrt{\log r}\ge 80,\;\;\text{for }r \ge R_4,
\end{equation}
and also so that $R_4\ge \max\{R_1,R_2,R_3\}$, where $R_1$ is the constant in
Lemma~\ref{Had}, $R_2$ is the constant in Corollary~\ref{Bohr}, and $R_3$ is the
constant in Lemma~\ref{annuli-N}.

In the proof we construct annuli $B_n$ of the form given in \eqref{Bn-form} which satisfy the conclusions of the theorem and also satisfy
\begin{equation}\label{prop1-a}
M(r^{k_n}_{n})>r^{k_{n+1}}_{n+1}>r_{n+1}\ge M(r_n) > (r^{k_n}_n)^2,\;\;\text{for } n\ge 0,
\end{equation}
a sequence of estimates that is closely related to \eqref{prop-2a}.

Take $r_0\ge R_4$. Then, by Lemma~\ref{annuli-N}, there exists $N(1)\ge 0$ and annuli
\[
A_n=A(r_n,r_n^{k_n}),\;\;n=0,\ldots, N(1),
\]
with the properties given in that lemma.

Putting
\begin{equation}\label{Bn-def}
B_n=\overline A_n,\;\;\text{for } n=0,\ldots,N(1),
\end{equation}
we deduce, by \eqref{prop-3}, that
\[
f(B_n)\supset B_{n+1},\;\;\text{for } n=0,\ldots, N(1)-1,
\]
and, by \eqref{prop-2a}, that
\begin{equation}\label{M-rn}
M(r^{k_n}_{n})>r^{k_{n+1}}_{n+1}>r_{n+1}=M(r_n)>(r^{k_n}_n)^2,\;\;\text{for } n=0,\ldots, N(1)-1.
\end{equation}
Also, by Lemma~\ref{annuli-N}~part~(b), for any $S,S',T,T'$ that satisfy
\begin{equation}\label{R-conds-2}
2<S<S',\;\;T<T'\le M(\tfrac13r_{N(1)}^{1+\delta_{N(1)}})\;\;\text {and }\;S'\le\tfrac12 T,
\end{equation}
we have
\begin{equation}\label{BN-cover}
f(B_{N(1)})\;\;\text{contains}\;\;  \overline A(S,S')\;\;
\text{or}\;\;\overline A(T,T').
\end{equation}

We apply the covering property \eqref{BN-cover} in two situations. First we take~$S$, $S'$, $T$,~$T'$ as follows:
\begin{equation}\label{Rdef}
S=M(r_{N(1)})\;\;\text{and}\;\;S'=S^{1+20\delta(S)},
\end{equation}
and
\begin{equation}\label{tildeRdef}
T=S^{1+40\delta(S)}\;\;\text{and}\;\;T'=T^{1+20\delta(T)}.
\end{equation}
This choice is possible since, by~\eqref{M(r)-a},
\[
\frac{T}{S'}=S^{20\delta(S)}=\exp\left(20\sqrt{\log S}\right)\ge 2,
\]
and, by Lemma~\ref{Had} part~(b), \eqref{Rdef}, \eqref{tildeRdef} and the fact that
$\sqrt{\log r_{N(1)}}\ge 5\log 3$,
\begin{eqnarray}\label{array-T'}
M(\tfrac13r_{N(1)}^{1+\delta_{N(1)}})&=&M(r_{N(1)}^{1+\delta_{N(1)}-\log 3/\log r_{N(1)}})\ge M(r_{N(1)}^{1+\tfrac45\delta_{N(1)}})\\
&\ge& M(r_{N(1)})^{1+\tfrac45\delta_{N(1)}}=S^{1+\tfrac45\delta_{N(1)}}\notag\\
&\ge& S^{(1+40\delta(S))(1+20\delta(T))} = T',\notag
\end{eqnarray}
since~\eqref{M(r)-a} implies that $\delta(S)\le \delta_{N(1)}/100\le 1/80$.

Thus \eqref{BN-cover} holds for these choices of $S,S',T,T'$.
Hence we can define the pair $(r_{N(1)+1},\delta_{N(1)+1})$ to be either $(S,\delta(S))$ or $(T,\delta(T))$, and then define
\[
k_{N(1)+1}=1+20\delta_{N(1)+1},
\]
and
\[
B_{N(1)+1}=\overline A(r_{N(1)+1},r_{N(1)+1}^{k_{N(1)+1}}),
\]
to ensure that~\eqref{prop2-a} holds for $n=N(1)$. Then~\eqref{prop1-a} also holds for this value of $n$; that is,
\[
M(r_{N(1)}^{k_{N(1)}})>r_{N(1)+1}^{k_{N(1)+1}}>r_{N(1)+1}\ge M(r_{N(1)}) > (r^{k_{N(1)}}_{N(1)})^2.
\]
Indeed, by the definition of $r_{N(1)+1}$, the definition of $S$ and \eqref{prop-2b},
\begin{equation}\label{k-N(1)}
r_{N(1)+1}\ge S=M(r_{N(1)})>(r_{N(1)}^{k_{N(1)}})^2,
\end{equation}
and, by Lemma~\ref{Had} part~(b) and the definition of $k_{N(1)+1}$,
\begin{eqnarray*}
M(r_{N(1)}^{k_{N(1)}})&\ge& M(r_{N(1)})^{k_{N(1)}}\\
&=&S^{k_{N(1)}}\ge S^{1+20\delta_{N(1)}}\\
&\ge& S^{(1+40\delta(S))(1+20\delta(S))}\ge T^{1+20\delta_{N(1)+1}}\\
&=& T^{k_{N(1)+1}}\ge r^{k_{N(1)+1}}_{N(1)+1},
\end{eqnarray*}
since $\delta_{N(1)+1}=\delta(r_{N(1)+1})\le \delta(S)\le\delta_{N(1)}/100\le 1/80$.

We also apply the covering property \eqref{BN-cover} when
\[
A(S,S')=A(r_p,r_p^{k_p})\;\;\text{and}\;\;A(T,T')=A(r_q,r_q^{k_q}),
\]
where $0\le p<q\le N(1)$, which is possible since
\[
r_n^{k_n}<r_{n+1}^{1/2},\;\;\text{so}\;\;r_{n+1}\ge 2r_n^{k_n},\;\;\text{for }n=0,\ldots, N(1)-1,
\]
by \eqref{M-rn}, and
\[
r_{N(1)}^{k_{N(1)}}\le M(\tfrac13r_{N(1)}^{1+\delta_N(1)}),
\]
by~\eqref{k-N(1)}, for example. It follows that
\[
f(B_{N(1)})\supset B_n,\;\;\text{for } n=0,\ldots, N(1),\;\;\text{with at most one exception}.
\]

Now, since
\[
B_{N(1)+1}=\overline A(r_{N(1)+1},r_{N(1)+1}^{k_{N(1)+1}}),\;\;\text{where }k_{N(1)+1}=1+20\delta_{N(1)+1},
\]
we can apply Lemma~\ref{annuli-N} with $r_0$ replaced by $r_{N(1)+1}$ to give $N(2)\in \N$ with $N(2)>N(1)$ and closed annuli $B_n$, $n=N(1)+1,\ldots, N(2)$, which satisfy \eqref{prop1-a} and \eqref{prop2-a} for $n=N(1)+1,\ldots, N(2)$ and also
\[
f(B_{N(2)})\supset B_n,\;\;\text{for } n=0,\ldots, N(2),\;\;\text{with at most one exception}.
\]

Repeating this application of Lemma~\ref{annuli-N} infinitely often, and introducing the subsequence $n_j=N(j)$, $j\in \N$, we obtain a sequence $(B_n)$ that satisfies \eqref{prop2-a}, \eqref{subseq2} and \eqref{prop1-a}. Also, since $f$ has no {\mconn} Fatou components, all the components of $J(f)$ are unbounded (see, for example, \cite[Theorem~1]{mK98}) and so the closed annuli $B_n$ must all meet $J(f)$ by \eqref{prop2-a}.

Finally, \eqref{prop1-a} implies that there exists $R=R(f)>0$ such that \eqref{prop1-bbb} holds. Indeed, by \eqref{prop1-a} we can choose $R>0$ such that
\[
r_0^{k_0}<R<r_1<r^{k_1}_1<M(R),
\]
so $B_0\subset A_0(R)$ and $B_1\subset A_1(R)$, and, by \eqref{prop1-a} again,
\[
M^{n-1}(R)<r_n<r^{k_n}_n<M^n(R),\;\;\text{so}\;\;B_n\subset A_n(R),\;\;\text{for }n\in\N,
\]
which proves \eqref{prop1-bbb}.
\end{proof}

\section{Proof of Theorem~\ref{itin}}\label{itin-proof}
\setcounter{equation}{0}
We deduce Theorem~\ref{itin} from Theorem~\ref{annuli} by using the following simple topological lemma, proved in \cite[Lemma~1]{RS08c}.
\begin{lemma}\label{top}
Let $E_n$, $n\in\N$, be a sequence of compact sets in $\C$ and
$f:\C\to\C$ be a continuous function such that
\begin{equation}\label{contains}
f(E_n)\supset E_{n+1},\quad\text{for } n\ge 0.
\end{equation}
Then there exists $\zeta$ such that $f^n(\zeta)\in E_n$, for $n\ge 0$.

If $f$ is also meromorphic and $E_n\cap J(f)\ne
\emptyset$, for $n\ge 0$, then there exists $\zeta\in J(f)$ such that
$f^n(\zeta)\in E_n$, for $n\ge 0$.
\end{lemma}

\begin{proof}[Proof of Theorem~\ref{itin}]
By Theorem~\ref{annuli}, there exists $R=R(f)>0$ and a sequence of closed annuli
\[
B_n=\overline A(r_n,r'_n),\;\;n\ge 0,
\]
each of which meets $J(f)$, such that
\begin{equation}\label{prop1-c}
B_n\subset A_n(R),\;\;\text{for } n\ge 0,
\end{equation}
\begin{equation}\label{prop3-c}
f(B_n)\supset B_{n+1},\;\;\text{for } n\ge 0,
\end{equation}
and there is a subsequence $B_{n_j}$ such that, for $j\in \N$,
\begin{equation}\label{prop4-c}
f(B_{n_j})\supset B_{n},\;\;\text{for } 0\le n\le n_j,\;\;\text{with at most one exception}.
\end{equation}

To prove Theorem~\ref{itin} we take this value of $R$, this sequence $(n_j)$ of non-negative integers and, for $j\in\N$, let $I_j$ be either the empty set or the singleton set consisting of the possible exceptional integer in $\{0,\ldots,n_j\}$ that occurs in~\eqref{prop4-c}.

If $s_0s_1\ldots$ is any sequence of non-negative integers satisfying the hypotheses of Theorem~\ref{itin}, then it follows from \eqref{prop3-c} and \eqref{prop4-c} that the compact sets
\[
E_n=B_{s_n}\subset A_{s_n}(R), \;\;\text{for }n\ge 0,
\]
satisfy \eqref{contains}. Then, by Lemma~\ref{top}, there exists a point $\zeta\in J(f)$ with itinerary $s_0s_1\ldots$. This proves the first part of Theorem~\ref{itin} and the last part follows immediately from the last part of Theorem~\ref{annuli}.
\end{proof}

We now deduce Corollary~\ref{itin-types}.
\begin{proof}[Proof of Corollary~\ref{itin-types}]
Parts~(a),~(b) and~(c) follow easily from Theorem~\ref{itin}, by using appropriate sequences $s_0s_1\ldots$ chosen to satisfy property \eqref{itin-prop} and $s_n\ge s$ for $n\ge 0$. The uncountable number of itineraries in parts~(b) and~(c) follows from the fact that if $j$ is such that $s_{n_j}\ge s+2$, then there are at least two possible choices for $s_{n_j+1}$ from the set $\{s,s+1,s+2\}$.

Part~(d) is proved as follows. Let $(a_n)$ be any positive sequence such that $a_n\to\infty$ as $n\to \infty$. Without loss of generality we can assume that $(a_n)$ is increasing. Then we use the fact that
\begin{equation}\label{B-2}
f^2(B_{n_j})\supset B_{n_j},\;\;\text{for }j\in\N,
\end{equation}
which follows from \eqref{prop3-c} and \eqref{prop4-c}, to choose a sequence $s_0s_1\ldots$ that has property \eqref{itin-prop} and is such that, for $j\in\N$,
\begin{equation}\label{Nj}
N_j=\max\{n:s_n=n_j\}\;\;\text{satisfies}\;\; a_{N_j} \ge M^{n_{j+1}}(R).
\end{equation}
The point $\zeta$ with itinerary $s_0s_1\ldots$ has an orbit that visits each annulus $B_{n_j}$ so often
that
\[
|f^n(\zeta)|\le M^{s_n}(R)\le a_n, \;\;\text{for }n\ge N_1.
\]
Once again we can obtain uncountably many such itineraries by using the fact, which follows from \eqref{B-2}, that for each $j\in \N$ there are infinitely many ways of choosing the integer $N_j$ to satisfy \eqref{Nj}.
\end{proof}

\section{Proof of Theorem~\ref{prescribed}}\label{prescribed-proof}
\setcounter{equation}{0}
We prove two versions of Theorem~\ref{prescribed} in the first of which it is assumed that~$f$ has a {\mconn} Fatou component.
\begin{theorem}\label{prescribed-MC}
Let $f$ be a {\tef} with a {\mconn} Fatou component. Then there exists $R_0=R_0(f)>0$ with the property that whenever $(a_n)$ is a positive sequence such that
\begin{equation}\label{aprops-MC}
a_n \ge R_0\;\;\text{and}\;\;a_{n+1}\le M(a_n),\;\;\text{for }n\ge 0,
\end{equation}
there is a point $\zeta\in J(f)$ and a sequence $(n_j)$ with $n_j\to\infty$ as $j\to \infty$ such that
\begin{equation}\label{presc-MC}
|f^n(\zeta)| \geq a_n, \mbox{ for } n\ge 0, \;\mbox{ and } |f^{n_j}(\zeta)| \le M^2(a_{n_j}), \mbox{ for } j \in \N.
\end{equation}

If, in addition,
\begin{equation}\label{notfast-MC}
\text{for each }\ell\in\N \;\text{there exists }n(\ell)\in\N \;\text{such that } a_{n(\ell)+\ell}<M^{n(\ell)}(R_0),
\end{equation}
then there are uncountably many itineraries with respect to $\{A_n(R_0)\}$ that correspond to points $\zeta$ satisfying~\eqref{presc-MC}.
\end{theorem}
{\it Remarks}\;\;1.\;It is natural to ask if the expression $M^2(a_{n_j})$ in~\eqref{presc-MC} can be replaced by $M(a_{n_j})$.

2.\; Note that the hypothesis \eqref{notfast-MC} is independent of $R_0$ for sufficiently large~$R_0$.
\begin{proof}[Proof of Theorem~\ref{prescribed-MC}]
Let $R=R(f)>0$ and $B_n=\overline A(t_n,t'_n)$, $n\ge 0$, be the constant and the sequence of closed annuli that were obtained in Theorem~\ref{annuli-MC}. These annuli all meet $J(f)$ and have the following properties:
\begin{equation}\label{prop1-e}
[t_{n+1},t'_{n+1}]\subset [M(t_n),M(t'_n)],\;\;\text{for } n\ge 0,
\end{equation}
\begin{equation}\label{prop1-d}
B_n\subset A_n(R),\;\;\text{for } n\ge 0,
\end{equation}
\begin{equation}\label{prop3-d}
f(B_n)\supset B_{n+1},\;\;\text{for } n\ge 0,
\end{equation}
and $(B_n)$ has a subsequence, here called $(B_{N_j})$, such that
\begin{equation}\label{prop3-e}
f(B_{N_j})\supset  B_{n},\;\;\text{for }j\in\N,\;0\le n\le N_j.
\end{equation}

Let $R_0=R$ and suppose that $(a_n)$ is any sequence that satisfies
\eqref{aprops-MC}. We shall apply Lemma~\ref{top} to a sequence of closed annuli
$(E_n)$, all of which are chosen from the sequence $(B_n)$ and satisfy
\[
E_n\subset \{z:|z|\ge a_n\},\;\;\text{for }n\ge 0.
\]

First define
\[
E_0=B_p,\;\;\text{where } p \text{ is the least integer such that } t_p\ge a_0.
\]
Then, for each $n\ge 0$, we choose $E_{n+1}$ depending on $E_n$ by applying the following rule.

Suppose that $E_n=B_m$ and $t_m\ge a_n$:

(1)\;\;if $m=N_j$ for some $j\in\N$, then
\begin{equation}\label{m-nj}
E_{n+1}=B_{p},\;\;\text{where }p \text{ is the least integer such that } t_p\ge
a_{n+1};
\end{equation}

(2)\;\;otherwise $E_{n+1}=B_{m+1}$.

Since \eqref{prop1-e} and \eqref{aprops-MC} guarantee that
\[
t_{m+1}\ge M(t_m)\ge M(a_n)\ge a_{n+1},
\]
we have $p\le m+1$ in case~(1) and, in either case,
\[
E_{n+1}\subset \{z:|z|\ge a_{n+1}\}.
\]
Also, in case~(1) the definition of~$p$ in \eqref{m-nj} gives $a_{n+1}>t_{p-1}$ (note that $p\ge 1$ because $a_{n+1}\ge R>t_0$ by \eqref{prop1-d}). Hence, by \eqref{prop1-d} and \eqref{prop1-e} again,
\[
M^2(a_{n+1})>M^2(t_{p-1})>M(t'_{p-1})>t'_p\,,
\]
so, in this case,
\[
E_{n+1}=B_p\subset \{z:|z|\le M^2(a_{n+1})\}.
\]

With this choice of the annuli $E_n$ it follows from \eqref{prop3-d} and
\eqref{prop3-e} that
\[
f(E_n)\supset E_{n+1},\;\;\text{for }n\ge 0,
\]
that
\[
E_n\cap J(f) \ne \emptyset\;\;\text{and}\;\;E_n\subset \{z:|z|\ge a_n\},\;\;\text{for }n\ge 0,
\]
and that $(E_n)$ has a subsequence, $(E_{n_j})$ say, satisfying
\[
E_{n_j}\subset \{z:|z|\le M^2(a_{n_j})\}.
\]

Therefore any $\zeta$ given by Lemma~\ref{top}, with the property that
\[
f^n(\zeta)\in E_n\cap J(f),\;\;\text{for }n\ge 0,
\]
satisfies~\eqref{presc-MC} with this subsequence $n_j$.

Finally suppose that \eqref{notfast-MC} holds. Then, for all $\ell\in\N$, we deduce, by \eqref{aprops-MC}, that
\begin{equation}\label{notfast-MC-1}
a_{n+\ell}<M^n(R_0),\;\;\text{for }n\ge n(\ell).
\end{equation}
We can obtain other points $\zeta$ that satisfy \eqref{presc-MC} with different subsequences $(n_j)$ and different itineraries by using the same construction as that given above but choosing $E_{n+1}$ in case~(1) as follows: whenever $E_n=B_m$, $m=N_j$ and the value of~$p$ specified by \eqref{m-nj} satisfies $p\le m$ we choose either the value of~$p$ that is specified by \eqref{m-nj} or $p=m+1$. The conditions \eqref{presc-MC} and \eqref{notfast-MC-1} imply that the value of~$p$ specified by \eqref{m-nj} satisfies $p\le m$ infinitely often (for otherwise $E_n=B_{n+q}$ for some $q\in\Z$ and all sufficiently large~$n$). If we vary the construction in this way and consider all possible ways of choosing $E_{n+1}$ in case~(1) that correspond to a point~$\zeta$ satisfying \eqref{presc-MC}, then the itineraries with respect to $\{A_n(R_0)\}$ that arise form the branches of an infinite binary tree and so are uncountable. This completes the proof of Theorem~\ref{prescribed-MC}.
\end{proof}
In our second version of Theorem~\ref{prescribed} it is assumed that~$f$ has no
{\mconn} Fatou components. Here a very similar technique to that used in the proof of Theorem~\ref{prescribed-MC} could be used to deduce the result from the annuli $B_n$
obtained in Theorem~\ref{annuli-noMC}, but using this approach the estimate for
$f^{n_j}(\zeta)$ would be only $|f^{n_j}(\zeta)|\le M^3(a_{n_j})$ because of the
possible exceptional annulus not covered by $f(B_{n_j})$ in
Theorem~\ref{annuli-noMC}. However, we can obtain a much stronger result by using
Lemma~\ref{annuli-N} directly.

In this proof we once again use the notation $\delta(r)=1/\sqrt{\log r}$.
\begin{theorem}\label{prescribed-noMC}
Let $f$ be a {\tef} with no multiply connected Fatou components and let $\eps\in (0,1)$. Then there exists $R_0=R_0(f,\eps)>0$ with the property that whenever $(a_n)$ is a positive sequence such that
\begin{equation}\label{aprops-noMC}
a_n \ge R_0\;\;\text{and}\;\;a_{n+1}\le M(a_n),\;\;\text{for }n\ge 0,
\end{equation}
there exists a point $\zeta\in J(f)$ and a sequence $(n_j)$ with $n_j\to\infty$ as $j\to \infty$ such that
\begin{equation}\label{presc-noMC}
|f^n(\zeta)| \geq a_n, \mbox{ for } n \ge 0, \;\mbox{ and } |f^{n_j}(\zeta)| \le a_{n_j}^{1+\eps}, \mbox{ for } j \in \N.
\end{equation}

If, in addition,
\begin{equation}\label{notfast-noMC}
\text{for each }\ell\in\N \;\text{there exists }n(\ell)\in\N \;\text{such that } a_{n(\ell)+\ell}<M^{n(\ell)}(R_0),
\end{equation}
then there are uncountably many itineraries with respect to $\{A_n(R_0)\}$ that correspond to points $\zeta$ satisfying~\eqref{presc-noMC}.
\end{theorem}
\begin{proof} First we choose $R_0=R_0(f,\eps)$ so large that $R_0\ge \max\{R_1,R_2,R_3\}$, where $R_1=R_1(f)$ is the constant in Lemma~\ref{Had}, $R_2=R_2(f)$ is the constant in Corollary~\ref{Bohr} and $R_3=R_3(f)$ is the constant in Lemma~\ref{annuli-N}, and also so that
\begin{equation}\label{R-eps}
M(r)> r^{10000}\;\;\text{and}\;\;\sqrt{\log r}\ge 100/\eps,\;\;\text{for } r\ge R_0.
\end{equation}

Suppose that $(a_n)$ is any sequence satisfying \eqref{aprops-noMC}. As in
the proof of Theorem~\ref{prescribed-MC}, the idea is to choose a suitable sequence of closed
annuli~$E_n$, related to the sequence $(a_n)$, and then apply Lemma~\ref{top}.

First define $r_0=a_1$. By Lemma~\ref{annuli-N}, there exists $N(1)\in\N$ and annuli
\[
A_n=A(r_n,r_n^{k_n}),\;\;n=0,\ldots, N(1),
\]
with the properties given in that lemma. In particular,
\[
r_{n+1}=M(r_n), \;\;\text{for }n=0,\ldots, N(1)-1,
\]
so, by \eqref{aprops-noMC},
\begin{equation}\label{rn-an}
r_n\ge a_n,\;\;\text{for }n=0,\ldots,N(1).
\end{equation} 

Also, with
\begin{equation}\label{En-def}
E_n=\overline A_n,\;\;\text{for } n=0,\ldots,N(1),
\end{equation}
we have
\begin{equation}\label{En-contains}
f(E_n)\supset E_{n+1},\;\;\text{for } n=0,\ldots, N(1)-1,
\end{equation}
and, for any $S,S',T,T'$ that satisfy
\begin{equation}\label{R-conds-3}
2<S<S',\;\;T<T'\le M(\tfrac13 r_{N(1)}^{1+\delta_{N(1)}})\;\;\text {and }\;S'\le \tfrac12T,
\end{equation}
we have
\begin{equation}\label{BN-cover-2}
f(E_{N(1)})\;\;\text{contains}\;\; \overline A(S,S')\;\;
\text{or}\;\; \overline A(T,T').
\end{equation}

We now apply the covering property \eqref{BN-cover-2} with
\begin{equation}\label{Rdef-2}
S=a_{N(1)+1}\;\;\text{and}\;\;S'=S^{1+20\delta(S)},
\end{equation}
and
\begin{equation}\label{tildeRdef-2}
T=S^{1+40\delta(S)}\;\;\text{and}\;\;T'=T^{1+20\delta(T)}.
\end{equation}
To see that this choice of $S, S',T, T'$ is possible first note that, by \eqref{aprops-noMC},~\eqref{R-eps} and \eqref{Rdef-2},
\[
\frac{T}{S'}=S^{20\delta(S)}=\exp\left(20\sqrt{\log
S}\right)\ge 2.
\]
Then, by using~\eqref{Rdef-2} and \eqref{tildeRdef-2}, together with the fact that
\[
S=a_{N(1)+1}\le M(a_{N(1)}) \le M(r_{N(1)}),
\]
which follows from \eqref{aprops-noMC} and \eqref{rn-an}, we deduce that
\begin{eqnarray*}
T'&=& S^{(1+40\delta(S))(1+20\delta(T))}\\
&\le & S^{(1+40\delta(S))(1+20\delta(S))}\\
&\le & M(r_{N(1)})^{\left(1+40\delta( M( r_{N(1)})\right)\left(1+20\delta(M(r_{N(1)})\right)}.
\end{eqnarray*}
The final inequality here holds because the function
\[
t\mapsto t^{\left(1+40/\sqrt{\log t}\right)\left(1+20/\sqrt{\log t}\right)}\;\;\text{is increasing on }[1,\infty).
\]
Since $\delta(M(r_{N(1)})\le \delta_{N(1)}/100\le 1/80$, by \eqref{R-eps}, we deduce
that
\[
\left(1+40\delta(M(r_{N(1)})\right)\left(1+20\delta(M(r_{N(1)})\right)\le
1+\tfrac45\delta_{N(1)},
\]
so
\[
T'\le M(r_{N(1)})^{1+\tfrac45\delta_{N(1)}}.
\]
As in the proof of \eqref{array-T'}, it now follows, by Lemma~\ref{Had} part~(b) and the fact that $\sqrt{\log r_{N(1)}}\ge 5\log 3$, that
\[
T'\le M(r_{N(1)})^{1+\tfrac45\delta_{N(1)}}\le
M(r_{N(1)}^{1+\tfrac45\delta_{N(1)}})\le M(\tfrac13 r_{N(1)}^{1+\delta_{N(1)}}).
\]
Thus \eqref{R-conds-3} holds with this choice of $S,S',T,T'$, so \eqref{BN-cover-2} does also. Moreover,
\begin{equation}\label{N(1)-est}
a_{N(1)+1}=S<T'=S^{(1+40\delta(S))(1+20\delta(T))}\le a_{N(1)+1}^{1+\eps},
\end{equation}
by \eqref{Rdef-2}, \eqref{tildeRdef-2}, \eqref{aprops-noMC} and \eqref{R-eps}.

Hence if we define the pair $(r_{N(1)+1},\delta_{N(1)+1})$ to be either $(S,\delta(S))$ or $(T,\delta(T))$, and put
\begin{equation}\label{E-next-1}
E_{N(1)+1}=\overline A(r_{N(1)+1},r_{N(1)+1}^{k_{N(1)+1}}),
\end{equation}
where
\[
k_{N(1)+1}=1+20\delta_{N(1)+1}=1+20\delta(r_{N(1)+1}),
\]
then, by \eqref{BN-cover-2},
\begin{equation}\label{E-next-2}
f(E_{N(1)})\supset E_{N(1)+1},
\end{equation}
and, by \eqref{N(1)-est},
\begin{equation}\label{En-2}
E_{N(1)+1}\subset \overline A(a_{N(1)+1},a_{N(1)+1}^{1+\eps}).
\end{equation}

Now, since $E_{N(1)+1}$ has the form in \eqref{E-next-1}, we can apply Lemma~\ref{annuli-N} with $r_0$ replaced by $r_{N(1)+1}$ to give $N(2)\in \N$ with $N(2)>N(1)$ and closed annuli
\[
E_n=\overline A(r_n,r_n^{k_n}),\;\; n=N(1)+1,\ldots, N(2),
\]
which satisfy
\[
r_n\ge a_n,\;\;\text{for }n=N(1)+1,\ldots,N(2),
\]
\[
f(E_n)\supset E_{n+1},\;\;\text{for }n=N(1)+1,\ldots, N(2)-1,
\]
\[
f(E_{N(2)})\supset \overline A(r_{N(2)+1},r_{N(2)+1}^{k_{N(2)+1}})=E_{N(2)+1},\;\;\text{say},
\]
where
\[
k_{N(2)+1}=1+20\delta_{N(2)+1}=1+20\delta(r_{N(2)+1})
\]
and
\[
E_{N(2)+1}\subset \overline A(a_{N(2)+1},a_{N(2)+1}^{1+\eps}).
\]
Repeating this application of Lemma~\ref{annuli-N} infinitely often, we obtain closed annuli
\[
E_n=\overline A(r_n,r_n^{k_n}),\;\;n\ge 0,
\]
such that
\begin{equation}\label{rn-an-2}
r_n\ge a_n,\;\;\text{for } n\ge 0,
\end{equation}
\begin{equation}\label{En-contains-2}
f(E_n)\supset E_{n+1},\;\;n\ge 0,
\end{equation}
and a strictly increasing sequence of positive integers $N(j)$, $j\in \N$, such that
\begin{equation}\label{En-subseq}
E_{N(j)+1}\subset \overline A(a_{N(j)+1},a_{N(j)+1}^{1+\eps}),\;\;\text{for }j\in\N.
\end{equation}
Since $f$ has no {\mconn} Fatou components, all the components of $J(f)$ are
unbounded, as noted earlier, so the closed annuli $E_n$ must all meet $J(f)$ by
\eqref{En-contains-2} and the fact that $J(f)$ is completely invariant under $f$.

By \eqref{En-contains-2} and Lemma~\ref{top}, there exists $\zeta$ with the property that
\[
f^n(\zeta)\in E_n\cap J(f),\;\;\text{for }n\ge 0.
\]
By \eqref{rn-an-2} and \eqref{En-subseq}, this $\zeta$ satisfies~\eqref{presc-noMC} with $n_j=N(j)+1$, $j\in\N$.

To prove the final part of Theorem~\ref{prescribed-noMC} we argue in a similar way to the proof of the final part of Theorem~\ref{prescribed-MC}. As in that proof, we deduce from \eqref{notfast-noMC} that for all $\ell\in\N$,
\begin{equation}\label{notfast-noMC-1}
a_{n+\ell}<M^n(R_0),\;\;\text{for }n\ge n(\ell).
\end{equation}
Next note that the covering property \eqref{BN-cover-2} can be used to show that, for any $j\in\N$,
\begin{equation}\label{forwards-1}
f(E_{N(j)})\supset\overline A(R',(R')^{k'}),
\end{equation}
where $R'$ and $k'$ satisfy
\begin{equation}\label{forwards-2}
R'\ge M(r_{N(j)})\;\;\text{and}\;\;k'=1+20\delta(R').
\end{equation}
The argument required here is similar to that used in the proof of Theorem~\ref{annuli-noMC} to show that $f(B_{N(j)})$ contains an annulus with these properties, and we omit the details.

Therefore in the above proof, for any $j\in\N$, instead of defining the annulus $E_{N(j)+1}$ by choosing the pair $(r_{N(j)+1},\delta_{N(j)+1})$ to be either $(S,\delta(S))$ or $(T,\delta(T))$, where the values of~$S$ and~$T$ are as given in \eqref{Rdef-2} and \eqref{tildeRdef-2} (with the suffix $N(1)+1$ replaced by $N(j)+1$), we can alternatively define
\begin{equation}\label{large-E}
E_{N(j)+1}=\overline A(r_{N(j)+1},r_{N(j)+1}^{k_{N(j)+1}}),
\end{equation}
where
\begin{equation}\label{large-r}
r_{N(j)+1}=R'\ge M(r_{N(j)})\;\;\text{and}\;\;k_{N(j)+1}=k'=1+20\delta(r_{N(j)+1}).
\end{equation}
Then, by \eqref{forwards-1} and \eqref{forwards-2}, we have $f(E_{N(j)})\supset E_{N(j)+1}$, and we can then continue the construction as before to obtain the itinerary of a point $\zeta$ that satisfies \eqref{presc-noMC}.

In view of \eqref{notfast-noMC-1}, as well as \eqref{large-E}, \eqref{large-r} and \eqref{aprops-noMC}, there must be infinitely many values of~$j$ for which these two possible choices of $E_{N(j)+1}$ lie in different elements of the partition $\{A_n(R_0)\}$. If we now consider all possible ways of choosing $E_{N(j)+1}$ that correspond to a point~$\zeta$ satisfying \eqref{presc-noMC}, then the itineraries with respect to $(A_n(R_0))$ that arise form the branches of an infinite binary tree and so are uncountable. This completes the proof of Theorem~\ref{prescribed-noMC}.
\end{proof}

\end{document}